\documentclass[11pt]{amsart}
\pdfoutput=1

\usepackage[utopia]{mathdesign}
\usepackage[protrusion=true,expansion=true]{microtype}
\usepackage[vmargin=1in,includeheadfoot,hmargin=1.25in,marginpar=.5in]{geometry}
\usepackage{setspace}
\usepackage[justification=RaggedRight,margin=1em]{subfig}

\usepackage{graphicx}
\usepackage{natbib}

\usepackage[pdftitle={Implicit-Explicit Variational Integration of
  Highly Oscillatory Problems}, pdfauthor={Ari Stern and Eitan
  Grinspun}]{hyperref}

\let\mathcal\relax
\DeclareMathAlphabet\mathcal{OMS}        {cmsy}{m}{n}

\newtheorem{theorem}{Theorem}[section]

\newtheorem{corollary}[theorem]{Corollary}
\newtheorem{proposition}[theorem]{Proposition}

\theoremstyle{definition}

\theoremstyle{remark}
\newtheorem*{remark}{Remark}

\begin{document}

\title{Implicit-Explicit Variational Integration of Highly Oscillatory Problems}

\author[A.~Stern]{Ari Stern}
\address{Ari Stern\\
Department of Applied and Computational Mathematics\\
California Institute of Technology\\
Pasadena, California 91125}
\curraddr{Department of Mathematics\\
University of California, San Diego\\
La Jolla, California 92093-0112}
\email{astern@math.ucsd.edu}
\thanks{First author's research partially supported by a Gordon and
  Betty Moore Foundation fellowship at Caltech, and by NSF grant
  CCF-0528101.}

\author[E.~Grinspun]{Eitan Grinspun}
\address{Eitan Grinspun\\
Department of Computer Science\\
Columbia University\\
New York, New York 10027}
\email{eitan@cs.columbia.edu} \thanks{Second author's research
  partially supported by the NSF (MSPA Award No. IIS-05-28402, CSR
  Award No. CNS-06-14770, CAREER Award No. CCF-06-43268).}



\dedicatory{}

\begin{abstract}
  In this paper, we derive a variational integrator for certain highly
  oscillatory problems in mechanics.  To do this, we take a new
  approach to the splitting of fast and slow potential forces: rather
  than splitting these forces at the level of the differential
  equations or the Hamiltonian, we split the two potentials with
  respect to the Lagrangian action integral.  By using a different
  quadrature rule to approximate the contribution of each potential to
  the action, we arrive at a geometric integrator that is implicit in
  the fast force and explicit in the slow force.  This can allow for
  significantly longer time steps to be taken (compared to standard
  explicit methods, such as St\"ormer/Verlet) at the cost of only a
  linear solve rather than a full nonlinear solve.  We also analyze
  the stability of this method, in particular proving that it
  eliminates the linear resonance instabilities that can arise with
  explicit multiple-time-stepping methods.  Next, we perform some
  numerical experiments, studying the behavior of this integrator for
  two test problems: a system of coupled linear oscillators, for which
  we compare against the resonance behavior of the r-RESPA method; and
  slow energy exchange in the Fermi--Pasta--Ulam problem, which
  couples fast linear oscillators with slow nonlinear oscillators.
  Finally, we prove that this integrator accurately preserves the slow
  energy exchange between the fast oscillatory components, which
  explains the numerical behavior observed for the Fermi--Pasta--Ulam
  problem.
\end{abstract}

\maketitle

\section{Introduction}

\subsection{Problem Background} Many systems in Lagrangian mechanics
have components acting on different time scales, posing a challenge
for traditional numerical integrators.  Examples include:
\begin{enumerate}
\item {\em Elasticity:} Several spatial elements of varying stiffness,
  resulting from irregular meshes and/or inhomogeneous
  materials~\citep{LeMaOrWe2003}.
\item {\em Planetary Dynamics:} $N$-body problem with nonlinear
  gravitational forces, arising from pairwise inverse-square
  potentials.  Multiple time scales result from the different
  distances between the bodies~\citep{FaBe2007}.
\item {\em Highly Oscillatory Problems:} Potential energy can be split
  into a ``fast'' linear oscillatory component and a ``slow''
  nonlinear component.  These problems are widely encountered in
  modeling molecular dynamics~\citep{LeReSk1996}, but have also
  been used to model other diverse applications, for example, in
  computer animation~\citep{EbEtHa2000,BoAs2004}.
\end{enumerate}
Because these systems each satisfy a Lagrangian variational principle,
they lend themselves readily to {\em variational integrators}: a class
of geometric numerical integrators designed for simulating Lagrangian
mechanical systems.  By construction, variational integrators preserve
a discrete version of this Lagrangian variational structure;
consequently, they are automatically symplectic and
momentum-conserving, with good long-time energy
behavior~\citep{MaWe2001}.

\subsubsection*{Explicit Methods, Multiple Time Stepping, and
  Resonance Instability.}
The St\"ormer/Verlet (or leapfrog) method is one of the canonical
examples of a geometric (and variational) numerical
integrator~\citep[see][]{HaLuWa2003}.  Yet, it and other simple,
explicit time stepping methods do not perform well for problems with
multiple time scales.  The maximum stable time step for these methods
is dictated by the stiffest mode of the underlying system; therefore,
the fastest force dictates the number of evaluations that must be
taken for {\em all} forces, despite the fact that the slow-scale
forces may be (and often are) much more expensive to evaluate.

To reduce the number of costly function evaluations associated to the
slow force, several explicit variational integrators use {\em multiple
  time stepping}, whereby different time step sizes are used to
advance the fast and slow degrees of freedom.  These include
substepping methods, such as Verlet-I/r-RESPA and mollified impulse,
where for each slow time step, an integer number of fast substeps are
taken~\citep{IzMaMaWiSlMoVi2002}.  More recently, asynchronous
variational integrators (AVIs) have been developed, removing the
restriction for fast and slow time steps to be integer (or even
rational) multiples of one another~\citep{LeMaOrWe2003}.
Multiple-time-stepping methods can be more efficient than
single-time-stepping explicit methods, like St\"ormer/Verlet, since
one can fully resolve the fast oscillations while taking many fewer
evaluations of the slow forces.  This is especially advantageous for
highly oscillatory problems, where the slow forces are nonlinear and
hence more computationally expensive to evaluate.

One drawback of multiple-time-stepping methods, however, is that they
can exhibit {\em linear resonance instability}.  This phenomenon
occurs when the slow impulses are nearly synchronized, in phase, with
the the fast oscillations.  These impulses artificially drive the
system at a resonant frequency, causing the energy (and hence the
numerical error) to increase without bound.  The problem of numerical
resonance is well known for substepping methods~\citep{BiSk1993}, and
has also recently been shown for AVIs as well---in fact, the subset of
fast and slow time step size pairs leading to resonance instability is
{\em dense} in the space of all possible
parameters~\citep{FoDaLe2007}.  Resonance instability can therefore be
difficult to avoid, particularly in highly oscillatory systems with
many degrees of freedom, as in molecular dynamics applications.

\subsubsection*{Implicit Methods for Single Time Stepping with Longer
  Step Sizes} Because multiple-time-stepping methods have these
resonance problems, a number of single-time-stepping methods have been
developed specifically for highly oscillatory problems.  As noted
earlier, single-time-stepping methods cannot fully resolve the fast
oscillations without serious losses in efficiency.  Therefore, the
goal of these methods is to take long time steps, {\em without}
actually resolving the fast oscillations, while still accurately
capturing the macroscopic behavior that emerges from the coupling
between fast and slow scales.  The challenge is to design methods that
allow for these longer time steps, without destroying either numerical
stability or geometric structure.

One obvious candidate integrator is the implicit midpoint method,
which is (linearly) unconditionally stable, as well as variational
(hence symplectic) and symmetric. Unfortunately, the stability of the
method comes at a cost: because the integrator is implicit in the slow
force, which is generally nonlinear, a nonlinear system of equations
must be solved at every time step. Therefore, just like the
fully-resolved St\"ormer/Verlet method, this means that the implicit
midpoint method requires an excessive number of function evaluations.

\subsubsection*{Implicit-Explicit Integration}
For highly oscillatory problems, {\em implicit-explicit} (IMEX)
integrators have been proposed as a potentially attractive alternative
to either explicit, multiple-time-stepping methods or implicit,
single-time-stepping methods.  Rather than using separate fast and
slow time step sizes, IMEX methods combine implicit integration (e.g.,
backward Euler) for the fast force with explicit integration (e.g.,
forward Euler) for the slow force.  Because the fast force is linear,
this semi-implicit approach requires only a linear solve for the
implicit portion, as opposed to the expensive nonlinear solve that
would be required for a fully implicit integrator, like the implicit
midpoint method.

IMEX methods were developed by~\citet{Crouzeix1980}, and have
continued to progress, including the introduction of IMEX Runge--Kutta
schemes for PDEs by~\citet{AsRuSp1997}.  However, in all of these
methods, the splitting is done at the level of the Euler--Lagrange
differential equations, rather than at the variational level of the
Lagrangian.  Consequently, a wide variety of IMEX schemes have been
created, both geometric and non-geometric, but in general they cannot
guarantee properties such as symplecticity, momentum conservation, or
good long-time energy behavior, which automatically hold for
variational integrators.  As an example of an IMEX integrator that is
not ``geometric'' in the usual sense, consider the LI and LIN methods
of~\citet{ZhSc1993}, which combine the backward Euler method with
explicit Langevin dynamics for molecular systems.  In particular, to
ameliorate the artificial numerical dissipation introduced by using
backward Euler, these methods rely on stochastic forcing to inject the
missing energy back into the system.

In this paper, we develop IMEX numerical integration from a
Lagrangian, variational point of view.  We do this by splitting the
fast and slow potentials at the level of the Lagrangian action
integral, rather than with respect to the differential equations or
the Hamiltonian. From this viewpoint, implicit-explicit integration is
an automatic consequence of discretizing the action integral using two
distinct quadrature rules for the slow and fast potentials.  The
resulting discrete Euler--Lagrange equations coincide with a
semi-implicit algorithm that was originally introduced
by~\citet{ZhSk1997} as a ``cheaper'' alternative to the implicit
midpoint method; \citet{AsRe1999a} also studied a variant of this
method for certain problems in molecular dynamics, replacing the
implicit midpoint step by the energy-conserving (but non-symplectic)
Simo--Gonzales method.

We also show that this variational IMEX method is free of resonance
instabilities; the proof of this fact is naturally developed at the
level of the Lagrangian, and does not require an examination of the
associated Euler--Lagrange equations. We then compare the
resonance-free behavior of variational IMEX to the
multiple-time-stepping method r-RESPA in a numerical simulation of
coupled slow and fast oscillators. Next, we evaluate the stability of
the variational IMEX method, for large time steps, in a computation of
slow energy exchange in the Fermi--Pasta--Ulam problem.  Finally, we
prove that the variational IMEX method accurately preserves this slow
energy exchange behavior (as observed in the numerical experiments) by
showing that it corresponds to a modified impulse method.

\subsection{A Brief Review of Variational Integrators}
The idea of variational integrators was studied by~\citet{Suris1990}
and \citet{MoVe1991}, among others, and a general theory was developed
over the subsequent decade~\citep[see][for a comprehensive
survey]{MaWe2001}.

Suppose we have a mechanical system on a configuration manifold $Q$,
specified by a Lagrangian $ L \colon TQ \rightarrow \mathbb{R} $.
Given a set of discrete time points $ t _0 < \cdots < t _N $ with
uniform step size $h$, we wish to compute a numerical approximation $
q _n \approx q \left( t _n \right),\ n = 0, \ldots , N $, to the
continuous trajectory $ q (t) $.  To construct a variational
integrator for this problem, we define a {\em discrete Lagrangian} $ L
_h \colon Q \times Q \rightarrow \mathbb{R} $, replacing tangent
vectors by pairs of consecutive configuration points, so that with
respect to some interpolation method and numerical quadrature rule we
have
\begin{equation*}
  L _h \left( q _n , q _{ n + 1 } \right) \approx \int_{ t _n} ^{ t _{ n +
      1 } } L \left( q, \dot{q} \right) \,dt .
\end{equation*}
Then the action integral over the whole time interval is approximated
by the {\em discrete action sum}
\begin{equation*}
  S _h [q] = \sum _{n=0}^{N-1}  L _h \left( q _n , q _{ n + 1 } \right) \approx
  \int_{ t _0} ^{ t _N } L \left( q, \dot{q} \right) \,dt .
\end{equation*}

If we apply Hamilton's principle to this action sum, so that $ \delta
S _h [q] = 0 $ when variations are taken over paths with fixed
endpoints, then this yields the {\em discrete Euler--Lagrange
  equations}
\begin{equation*}
  D _1 L _h \left( q _n , q _{ n + 1 } \right) + D _2 L _h \left( q _{ n - 1 } , q
    _n \right) = 0 , \qquad n = 1, \ldots, N - 1 ,
\end{equation*}
where $ D _1 $ and $ D _2 $ denote partial differentiation in the
first and second arguments, respectively.  This defines a two-step
numerical method on $ Q \times Q $, mapping $ \left( q _{ n - 1 } , q
  _n \right) \mapsto \left( q _n , q _{ n + 1 } \right) $.  The
equivalent one-step method on the cotangent bundle $T^\ast Q$, mapping
$ \left( q _n , p _n \right) \mapsto \left( q _{ n + 1 } , p _{ n + 1
  } \right) $, is defined by the {\em discrete Legendre transform}
\begin{equation*}
p _n = - D _1 L _h \left( q _n , q _{ n + 1 } \right) , \qquad p _{ n
  + 1 } = D _2 L _h \left( q _n , q _{ n + 1 } \right) ,
\end{equation*}
where the first equation updates $q$, and the second updates $p$.

\subsubsection*{Examples.}  Consider a Lagrangian of the form $ L
\left( q , \dot{q} \right) = \frac{1}{2} \dot{q} ^T M \dot{q} - V (q)
$, where $ Q = \mathbb{R}^{d} $, $M$ is a constant $ {d} \times {d} $
mass matrix, and $V \colon Q \rightarrow \mathbb{R} $ is a potential.
If we use linear interpolation of $q$ with trapezoidal quadrature to
approximate the contribution of $V$ to the action integral, we get
\begin{equation*}
  L _h ^\text{trap} \left( q _n , q _{ n + 1 } \right) = \frac{ h }{ 2 } \left(
    \frac{ q _{ n + 1 } - q _n }{ h } \right) ^T M \left(
    \frac{ q _{ n + 1 } - q _n }{ h } \right) - h \frac{ V
    \left( q _n \right) + V \left( q _{ n + 1 } \right) }{ 2 }  ,
\end{equation*}
which we call the {\em trapezoidal discrete Lagrangian}.  It is
straightforward to see that the discrete Euler--Lagrange equations for
$ L _h ^\text{trap} $ correspond to the explicit St\"ormer/Verlet
method.  Alternatively, if we use midpoint quadrature to approximate
the integral of the potential, this yields the {\em midpoint discrete
  Lagrangian},
\begin{equation*}
  L _h ^\text{mid} \left( q _n , q _{ n + 1 } \right) = \frac{ h }{ 2 } \left(
    \frac{ q _{ n + 1 } - q _n }{ h } \right) ^T M \left(
    \frac{ q _{ n + 1 } - q _n }{ h } \right) - h V  \left( \frac{ q
      _n +  q _{ n + 1 } }{ 2 } \right)  ,
\end{equation*}
for which the resulting integrator is the implicit midpoint method.

\section{A Variational IMEX Method}

In this section, we show how to develop a variational integrator that
combines aspects of the St\"ormer/Verlet and implicit midpoint methods
mentioned above.  The main idea is that, given a splitting of the
potential energy into fast and slow components, we define the discrete
Lagrangian by applying the midpoint quadrature rule to the fast
potential and the trapezoidal quadrature rule to the slow potential.
The resulting variational integrator is implicit in the fast force and
explicit in the slow force.  After this, we focus on the specific case
of highly oscillatory problems, where the fast potential is quadratic
(corresponding to a linear fast force).  In this case, we show that
the IMEX integrator can be understood as St\"ormer/Verlet with a
modified mass matrix.

It should be noted that the results in this section can also be
verified directly in terms of the numerical algorithm, and do not
strictly require making use of the Lagrangian variational structure.
However, we find the variational perspective to be useful and
illustrative, both in arriving at this particular IMEX algorithm and
in interpreting its numerical features.

\subsection{The IMEX Discrete Lagrangian and Equations of Motion}
Suppose that we have a Lagrangian of the form $ L \left( q, \dot{q}
\right) = \frac{1}{2} \dot{q} ^T M \dot{q} - U (q) - W (q) $, where
$U$ is a slow potential and $W$ is a fast potential, for the
configuration space $ Q = \mathbb{R}^{d} $.  Then define the {\em IMEX
  discrete Lagrangian}
\begin{equation*}
  L _h ^\text{IMEX} \left( q _n , q _{ n + 1 } \right) = \frac{ h }{ 2 } \left(
    \frac{ q _{ n + 1 } - q _n }{ h } \right) ^T M \left(
    \frac{ q _{ n + 1 } - q _n }{ h } \right) - h \frac{ U
    \left( q _n \right) + U \left( q _{ n + 1 } \right) }{ 2 } - h W
  \left( \frac{ q _n +  q _{ n + 1 } }{ 2 } \right)  ,
\end{equation*}
using (explicit) trapezoidal approximation for the slow potential and
(implicit) midpoint approximation for the fast potential.  The
discrete Euler--Lagrange equations give the two-step variational
integrator on $ Q \times Q $
\begin{equation*}
q _{ n + 1 } - 2 q _n + q _{ n - 1 } = - h ^2 M ^{-1} \left[ \nabla U
  \left( q _n \right) + \frac{1}{2} \nabla W \left( \frac{ q _{ n - 1
      } + q _n }{ 2 } \right) + \frac{1}{2} \nabla W \left( \frac{ q
      _n + q _{ n + 1 } }{ 2 } \right) \right] ,
\end{equation*}
and the corresponding discrete Legendre transform is given by
\begin{align*}
  p _n &= M \left( \frac{ q _{ n + 1 } - q _n }{ h } \right) + \frac{
    h }{ 2 } \nabla U \left( q _n \right) + \frac{ h }{ 2 } \nabla W
  \left( \frac{ q _n + q _{ n + 1 } }{ 2 } \right), \\
  p _{n+1} &= M \left( \frac{ q _{ n + 1 } - q _n }{ h } \right) - \frac{
    h }{ 2 } \nabla U \left( q _{n+1} \right) - \frac{ h }{ 2 } \nabla
  W \left( \frac{ q _n + q _{ n + 1 } }{ 2 } \right).
\end{align*}

To see how this translates into an algorithm for a one-step integrator
on $T^\ast Q$, it is helpful to introduce the intermediate stages
\begin{equation*}
  p _n ^+ = p _n - \frac{ h }{ 2 } \nabla U \left( q _n \right) ,
  \qquad  p ^- _{n+1} = p _{n+1} + \frac{ h }{ 2 } \nabla U \left( q
    _{n+1} \right).
\end{equation*}
Substituting these into the previous expression and rearranging yields
the algorithm
\begin{alignat*}{2}
  \textbf{Step 1:}&& p _n ^+ &= p _n - \frac{ h }{ 2 } \nabla U \left(
    q _n \right), \\
  \textbf{Step 2:} &\quad& \left\{ \hspace{-10.5em} \begin{aligned}
      \phantom{q _n + h M ^{-1} \left( \frac{ p _n ^+ + p _{ n + 1
            } ^- }{ 2 } \right)} q _{ n + 1 } &\\
      \phantom{p _n ^+ - h \nabla W \left( \frac{ q _n + q _{ n + 1
            }}{2} \right)} p _{ n + 1 } ^- &
  \end{aligned} \right. 
&\!
\begin{aligned}
&= q _n + h M ^{-1} \left( \frac{ p _n ^+ + p _{ n + 1
      } ^- }{ 2 } \right) ,\\
 &= p _n ^+ - h \nabla W \left( \frac{ q _n + q _{ n
        + 1 }}{2} \right) ,
\end{aligned}
\\
\textbf{Step 3:}&& p _{n+1} &= p _{n+1}^- - \frac{ h }{ 2 } \nabla U \left( q _{n+1}
\right) ,
\end{alignat*}
where Step 2 corresponds to a step of the implicit midpoint method.

This can be summarized, in the style of impulse methods, as:
\begin{enumerate}
\item {\bf kick:} explicit kick from $U$ advances $ \left( q _n , p _n
  \right) \mapsto \left( q _n , p _n ^+ \right) $,
\item {\bf oscillate:} implicit midpoint method with $W$ advances $
  \left( q _n, p _n ^+ \right) \mapsto \left( q _{ n + 1 } , p _{ n +
      1 } ^- \right) $,
\item {\bf kick:} explicit kick from $U$ advances $ \left( q _{n+1} ,
    p _{n+1}^- \right) \mapsto \left( q _{n+1} , p _{n+1} \right) $.
\end{enumerate}
In particular, notice that this reduces to the St\"ormer/Verlet method
when $ \nabla W \equiv 0 $ and to the implicit midpoint method when $
\nabla U \equiv 0 $.  Also, if the momentum $ p _n $ does not actually
need to be recorded at the full time step (i.e., collocated with the
position $ q _n $), then Step 3 can be combined with Step 1 of the
next iteration to create a staggered ``leapfrog'' method.

\subsubsection*{Interpretation as a Hamiltonian Splitting Method}
This algorithm on $ T ^\ast Q $ can also be interpreted as a fast-slow
splitting method~\citep[II.5 and VIII.4.1]{McQu2002,HaLuWa2006} for
the separable Hamiltonian $ H = T + U + W $, where $T$ is the kinetic
energy, as follows.  Let $ \Phi ^{ T+W } _h \colon T^\ast Q
\rightarrow T^\ast Q $ denote the numerical flow of the implicit
midpoint method with time step size $h$, applied to the fast portion
of the Hamiltonian $ T + W $, and let $ \varphi ^U _h \colon T^\ast Q
\rightarrow T^\ast Q $ be the exact Hamiltonian flow for the slow
potential $U$ (i.e., constant acceleration without displacement).
Then the variational IMEX method has the flow map $ \Psi _h \colon
T^\ast Q \rightarrow T^\ast Q $, which can be written as the following
composition of exact and numerical flows:
\begin{equation*}
  \Psi _h = \varphi  _{ h/2 } ^U  \circ \Phi _h ^{ T+W } \circ
  \varphi  ^U  _{ h/2 } .
\end{equation*}
This formulation highlights the fact that variational IMEX is
symmetric (since it is a symmetric composition of symmetric methods)
as well as symplectic (since it can be written as a composition of
symplectic maps).

\subsection{Application to Highly Oscillatory Problems} For highly
oscillatory problems on $ Q = \mathbb{R}^{d} $, we start by taking a
quadratic fast potential
\begin{equation*}
  W (q) = \frac{1}{2} q ^T \Omega ^2 q , \qquad \Omega \in \mathbb{R}
  ^{ {d} \times {d} } \text{ symmetric and positive semidefinite}.
\end{equation*}
A prototypical $\Omega$ is given by the block-diagonal matrix $ \Omega
= \left( \begin{smallmatrix}
    0 & 0 \\
    0 & \omega I
  \end{smallmatrix} \right) $, where some of the degrees of freedom
are subjected to an oscillatory force with constant fast frequency $
\omega \gg 1 $.  We also denote the slow force $ g (q) = - \nabla U
(q) $ and assume, without loss of generality, that the constant mass
matrix is given by $ M = I $.  Therefore, the nonlinear system we wish
to approximate numerically is
\begin{equation*}
\ddot{q} + \Omega ^2 q = g (q) .
\end{equation*}
This is the conventional setup for highly oscillatory problems, used
by~\citet[XIII]{HaLuWa2006} and others.

Applying the IMEX method to this example, we get the discrete
Lagrangian
\begin{align*} 
  L _h ^\text{IMEX} \left( q _n , q _{ n + 1 } \right) &= \frac{ h }{
    2 } \left( \frac{ q _{ n + 1 } - q _n }{ h } \right) ^T \left(
    \frac{ q _{
        n + 1 } - q _n }{ h } \right)\\
  &\quad - h \frac{ U \left( q _n \right) + U \left( q _{ n + 1 }
    \right) }{ 2 } - h \left( \frac{ q _n + q _{ n + 1 } }{ 2 }
  \right) ^T \Omega ^2 \left( \frac{ q _n + q _{ n + 1 } }{ 2 }
  \right),
\end{align*} 
and so the two-step IMEX scheme is given by the discrete
Euler--Lagrange equations
\begin{equation*}
  q _{ n + 1 } - 2 q _n + q _{ n - 1 } + \frac{ h ^2 }{ 4 } \Omega ^2
  \left( q _{ n + 1 } + 2 q _n + q _{ n - 1 } \right) = h ^2 g \left(
    q _n \right) .
\end{equation*}
Combining terms, we can rewrite this as
\begin{equation*}
  \left[ I + \frac{ h ^2 }{ 4 } \Omega ^2 \right] \left( q _{ n + 1 } -
    2 q _n + q _{ n - 1 } \right) + h ^2 \Omega ^2 q _n = h ^2 g \left(
    q _n \right) ,
\end{equation*}
which is equivalent to St\"ormer/Verlet with a modified mass matrix $
I + (h \Omega / 2) ^2 $.  This equivalence can similarly be shown to
hold for the one-step formulation of the IMEX scheme on $T^\ast
Q$---that is, the two methods also produce the same $ p _n $, as well
as the same $ q _n $.

In fact, this correspondence between IMEX and a modified
St\"ormer/Verlet method is true not just in the discrete
Euler--Lagrange equations, but in the discrete Lagrangian itself.  This
follows immediately from the following proposition.
\label{sec:imex-equiv-sv}
\begin{proposition}
  Suppose we have a Lagrangian $ L \left( q , \dot{q} \right) =
  \frac{1}{2} \dot{q} ^T M \dot{q} - \frac{1}{2} q ^T \Omega ^2 q $
  and its corresponding midpoint discrete Lagrangian $ L ^\text{mid}
  _h $.  Next, define the modified Lagrangian $ \tilde{L} \left( q ,
    \dot{q} \right) = \frac{1}{2} \dot{q} ^T \tilde{M} \dot{q} -
  \frac{1}{2} q ^T \Omega ^2 q $, having the same quadratic potential
  but a different mass matrix $ \tilde{ M } $, and take its
  trapezoidal discrete Lagrangian $ \tilde{L} ^\text{trap} _h $.  Then
  $ L ^\text{mid} _h \equiv \tilde{L} ^\text{trap} _h $ when $ \tilde{
    M } = M + \left( h \Omega / 2 \right) ^2 $.
\end{proposition}

\begin{proof}
The midpoint discrete Lagrangian is given by
\begin{equation*} 
  L _h ^\text{mid} \left( q _n , q _{ n + 1 } \right) = \frac{h}{2} \left(
    \frac{ q _{ n + 1 } - q _n }{ h } \right) ^T M \left( \frac{ q
      _{ n + 1 } - q _n }{ h } \right) - \frac{ h }{ 2 } \left( \frac{
      q _n + q _{ n  + 1 } }{ 2 } \right) ^T \Omega ^2 \left( \frac{ q
      _n + q _{ n + 1 } }{ 2 } \right).
\end{equation*} 
Now, notice that we can rearrange the terms
\begin{align*}
  - \left( \frac{ q _n + q _{ n + 1 } }{ 2 } \right) ^T \Omega ^2
  \left( \frac{ q _n + q _{ n + 1 } }{ 2 } \right) &= \left( \frac{ q
      _{ n + 1 } - q _n }{ 2 } \right) ^T \Omega ^2 \left( \frac{ q _{
        n + 1 } - q _n }{ 2 } \right) - \frac{1}{2} q _n ^T \Omega ^2
  q _n - \frac{1}{2} q _{ n + 1 } ^T \Omega ^2 q _{ n + 1 } \\
  &= \left( \frac{ q _{ n + 1 } - q _n }{ h } \right) ^T \left( \frac{
      h \Omega }{ 2 } \right) ^2 \left( \frac{ q _{ n + 1 } - q _n }{
      h } \right) - \frac{1}{2} q _n ^T \Omega ^2 q _n - \frac{1}{2} q
  _{ n + 1 } ^T \Omega ^2 q _{ n + 1 } .
\end{align*}
Therefore the discrete Lagrangian can be written in the trapezoidal
form
\begin{equation*} 
  L _h ^\text{mid} \left( q _n , q _{ n + 1 } \right) = \frac{h}{2}
  \left( \frac{ q _{ n + 1 } - q _n }{ h } \right) ^T \left[ M +
    \left( \frac{ h \Omega }{ 2 } \right) ^2 \right] \left( \frac{ q
      _{ n + 1 } - q _n }{ h } \right) - \frac{ h }{ 2 } \left(
    \frac{1}{2} q _n ^T \Omega ^2 q _n + \frac{1}{2} q _{ n + 1 } ^T
    \Omega ^2 q _{ n + 1 } \right) ,
\end{equation*} 
which is precisely $ \tilde{L} _h ^\text{trap} \left( q _n , q _{ n +
    1 } \right) $ when $ \tilde{ M } = M + \left( h \Omega / 2 \right)
^2 $.
\end{proof}

\begin{corollary}
  Consider a highly oscillatory system with an arbitrary slow
  potential $U$, quadratic fast potential $ W (q) = \frac{1}{2} q ^T
  \Omega ^2 q $, and constant mass matrix $ M = I $, so that the
  Lagrangian $L$ and IMEX discrete Lagrangian $ L _h ^\text{IMEX} $
  are defined as above.  Next, take the modified Lagrangian $ \tilde{
    L } $ with the same potentials but different mass matrix $ \tilde{
    M } $.  Then $ L _h ^\text{IMEX} \equiv \tilde{ L } _h
  ^\text{trap} $ when $ \tilde{M} = I + \left( h \Omega / 2 \right) ^2
  $.
\end{corollary}

\subsection{Analysis of Linear Resonance Stability}
\label{sec:resonance-stability}
To study the linear resonance stability of this IMEX integrator, we
consider a model problem where $U$ and $W$ both correspond to linear
oscillators.  Let $ U (q) = \frac{1}{2} q ^T q $ and $ W (q) =
\frac{1}{2} q ^T \Omega ^2 q $, where $ \Omega = \omega I $ for some $
\omega \gg 1 $, and again let the mass matrix $ M = I $.  Although
this is something of a ``toy problem''---obviously, one could simply
combine $U$ and $W$ into a single quadratic potential $ \frac{1}{2}
\left( 1 + \omega ^2 \right) q ^T q $---it is illustrative for
studying the numerical resonance of multiple-time-stepping methods,
since the system has no external forcing terms and hence no real {\em
  physical} resonance.

To prove that the IMEX method does not exhibit linear resonance
instability, we show that the stability condition only requires that
the time step be stable for the explicit slow force, and is
independent of the fast frequency $\omega$.  The idea of the proof is
to use the results from~\autoref{sec:imex-equiv-sv}, showing that
the IMEX method is equivalent to St\"ormer/Verlet with a modified mass
matrix, and then to apply the well-known stability criteria for
St\"ormer/Verlet.

In particular, for a harmonic oscillator with unit mass and frequency
$\nu$, the St\"ormer/Verlet method is linearly stable if and only if $
\left\lvert h \nu \right\rvert \leq 2 $, as can be shown by a
straightforward calculation of the eigenvalues of the propagation
matrix~\citep[p.~23]{HaLuWa2006}.  For a system with constant mass $m$
and spring constant $ \nu ^2 $, this condition generalizes to $ h ^2
\nu ^2 \leq 4 m $.

\begin{theorem}
\label{thm:imex-no-resonance}
  The IMEX method is linearly stable, for the system described above,
  if and only if $ h \leq 2 $ (i.e., if and only if $h$ is a stable
  time step size for the slow oscillator alone).
\end{theorem}

\begin{proof}
  As proved in the previous section, the IMEX method for this system
  is equivalent to St\"ormer/Verlet with the modified mass matrix $
  I + (h \Omega /2 )^2 $.  Now, this modified oscillatory system has
  constant mass $ m = 1 + \left( h \omega / 2 \right) ^2 $ and spring
  constant $ \nu ^2 = 1 + \omega ^2 $.  Therefore, the necessary and
  sufficient condition for linear stability is
  \begin{equation*}
    h ^2 \left( 1 + \omega ^2 \right) \leq 4 \left( 1 +
      \frac{ h ^2 }{ 4 } \omega ^2 \right) ,
\end{equation*}
and since the $ h ^2 \omega ^2 $ terms cancel on both sides, this is
equivalent to $ h ^2 \leq 4 $, or $ h \leq 2 $.
\end{proof}
This shows that, in contrast to multiple-time-stepping methods, the
IMEX method does not exhibit linear resonance instability.  In
particular, one can interpret the modified mass matrix as giving the
system an effective frequency of $ \sqrt{ \left( 1 + \omega ^2 \right)
  \bigl/ \left( 1 + \left( h \omega / 2 \right) ^2 \right) }$, which
attenuates the destabilizing high frequencies in the original system.
It should be noted that {\em nonlinear} instability is known to be
possible for the implicit midpoint method, although even that can be
avoided with a time step size restriction that is considerably weaker
than that required for explicit methods~\citep[see][]{AsRe1999}.

\section{Numerical Experiments}

\subsection{Coupled Linear Oscillators} To illustrate the numerical
resonance behavior of the variational IMEX scheme, as compared with a
multiple-time-stepping method, we consider the model problem
from~\autoref{sec:resonance-stability} for dimension $ {d} = 1 $
(i.e., $ Q = \mathbb{R} $).  \autoref{fig:resonance} shows a log plot
of the maximum absolute error in total energy (i.e., the Hamiltonian)
for both r-RESPA and the variational IMEX method, for a range of
frequencies $\omega$.  MATLAB simulations were performed over the time
interval $ [0,1000] $, with fixed time step size $ h = 0.1 $, and with
the normalized frequency $ \omega h / \pi $ ranging over $ ( 0, 4.5 ]
$.  Additionally, to fully resolve the fast oscillations, r-RESPA took
100 fast substeps of size $ h/100 = 0.001 $ for each full time step of
size $h$.

\begin{figure}
\includegraphics[width=.6\linewidth]{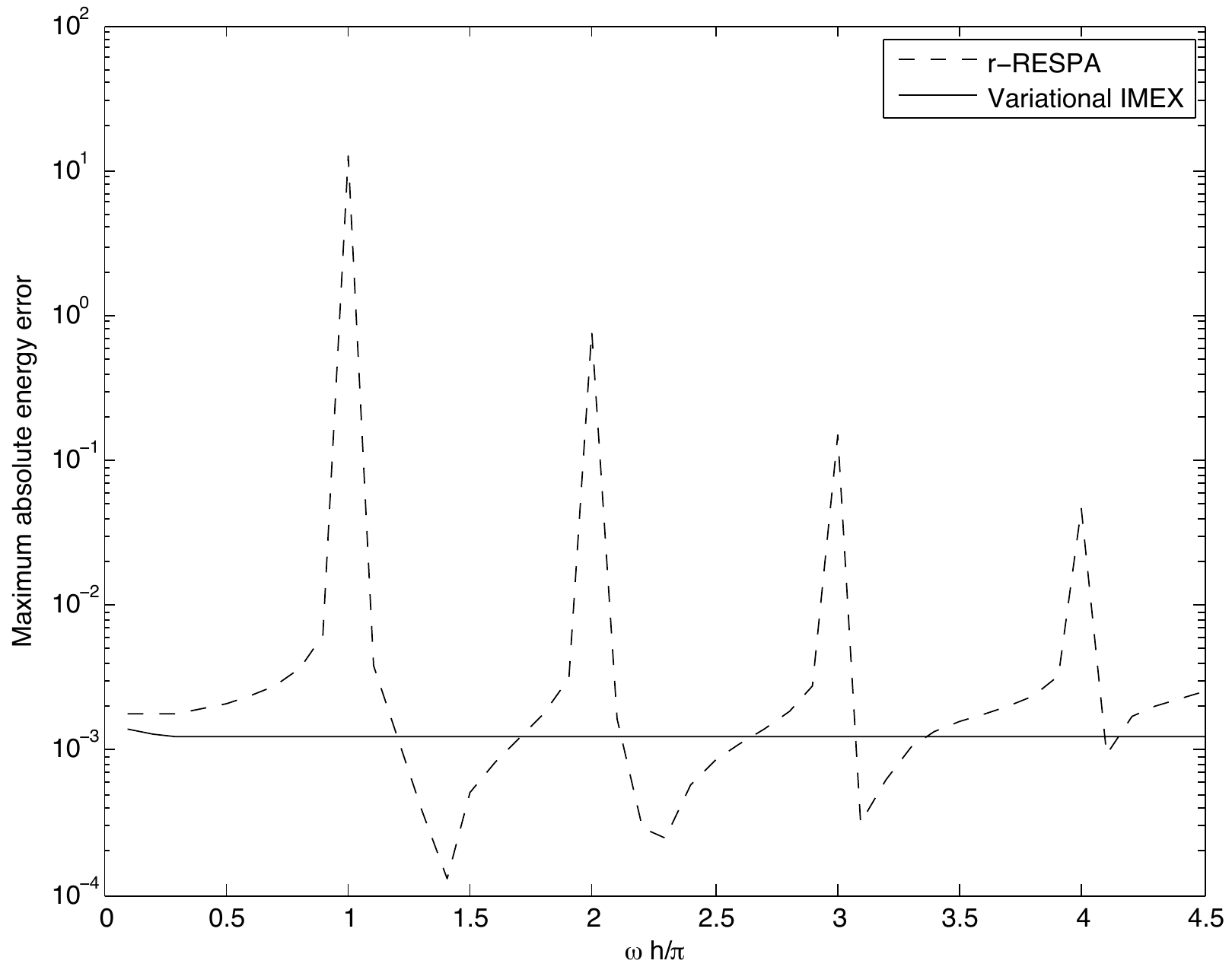}
\caption{Maximum energy error of r-RESPA and variational IMEX,
  integrated over the time interval $ [0,1000] $ for a range of
  parameters $\omega$.  The r-RESPA method exhibits resonance
  instability near integer values of $ \omega h / \pi $, while the
  variational IMEX method remains stable.}
\label{fig:resonance}
\end{figure}

The r-RESPA method exhibits ``spikes'' in the total energy error near
integer values of $ \omega h / \pi $, corresponding to the parameters
where resonance instability develops and the numerical solution
becomes unbounded.  (The finite size of these spikes is due to the
fact that the numerical simulation was run only for a finite interval
of time.  Interestingly, one also sees ``negative spikes,'' where the
fast and slow oscillations are exactly out-of-phase and cancel one
another.)  It should be noted that the small substep size of r-RESPA
is sufficient for stable integration of the fast force alone; it is
only the introduction of the slow force that makes things unstable.
By contrast, the maximum energy error for the variational IMEX method
is nearly constant for all values of $\omega$, showing no sign of
resonance.  This is fully consistent with the theoretical result
obtained in~\autoref{thm:imex-no-resonance}.

\subsection{The Fermi--Pasta--Ulam Problem}
As an example of a nontrivial highly oscillatory problem with
nonlinear slow potential, we chose the modified Fermi--Pasta--Ulam
(FPU) problem considered by~\citet[I.5 and XIII]{HaLuWa2006}, whose
treatment we will now briefly review.  The FPU problem consists of $ 2
\ell $ unit point masses, which are chained together, in series, by
alternating weak nonlinear springs and stiff linear springs.  (This
particular setup is due to~\citealp{GaGiMaVa1992}, and is a variant of
the problem originally introduced by~\citealp{FePaUl1955}.)  Clearly,
this system becomes rather trivial if we make the stiff springs
``infinitely stiff,'' replacing them by rigid constraints (as done by
some numerical methods, such as SHAKE/RATTLE).  However, for finite
stiffness, the FPU system exhibits interesting dynamics due to the
coupling between fast and slow springs.

Let us denote the displacements of the point masses by $ q _1 ,
\ldots, q _{2 \ell} \in \mathbb{R} $ (where the endpoints $ q _0 = q
_{ 2\ell + 1 } = 0 $ are taken to be fixed), and their conjugate
momenta by $ p _i = \dot{q} _i $ for $ i = 1, \ldots , 2\ell $.  In
these variables, the FPU system has the Hamiltonian
\begin{equation*}
H \left( q, p \right) = \frac{1}{2} \sum _{ i = 1 } ^\ell \left( p ^2
  _{ 2 i - 1 } + p ^2 _{ 2 i } \right) + \frac{ \omega ^2 }{ 4 } \sum
_{ i = 1 } ^\ell \left( q _{ 2i } - q _{ 2i - 1 } \right) ^2 + \sum _{
  i = 0 } ^\ell \left( q _{ 2i + 1 } - q _{ 2i } \right) ^4 ,
\end{equation*}
which contains a quadratic potential for the $ \ell $ stiff linear
springs, each with frequency $\omega$, and a quartic potential for the
$ \ell + 1 $ soft nonlinear (cubic) springs.  However, it is helpful
to perform the coordinate transformation \citep[following][p. 22]{HaLuWa2006}
\begin{align*}
  x _{ 0, i } &= \frac{ q _{ 2 i } + q _{ 2 i - 1 } }{ \sqrt{2} }, & x
  _{ 1,i } &= \frac{ q _{ 2 i } - q _{ 2 i - 1 } }{\sqrt{2} },\\
  y _{ 0, i } &= \frac{ p _{ 2 i } + p _{ 2 i - 1 } }{ \sqrt{2}}, &
  y _{ 1,i } &= \frac{ p _{ 2 i } - p _{ 2 i - 1 } }{ \sqrt{2} },
\end{align*}
so that (modulo rescaling) $ x _{ 0, i } $ corresponds to the location
of the $i$th stiff spring's center, $ x _{ 1, i } $ corresponds to its
length, and $ y _{ 0, i } , y _{ 1, i } $ are the respective conjugate
momenta.  Writing the Hamiltonian in these new variables, we have
\begin{align*} 
  H (x, y) &= \frac{1}{2} \sum _{ i = 1 } ^\ell \left( y _{ 0,i } ^2 +
    y _{ 1,i } ^2 \right) + \frac{ \omega ^2 }{ 2 } \sum _{ i = 1 }
  ^\ell x _{ 1,i } ^2 \\ & \quad + \frac{ 1 }{ 4 } \left[ \left( x _{
        0, 1 } - x _{ 1,1 } \right) ^4 + \sum _{ i = 1 } ^{\ell - 1 }
    \left( x _{ 0, i + 1 } - x _{ 1 , i + 1 } - x _{ 0, i } - x _{ 1 ,
        i } \right) ^4 + \left( x _{ 0,\ell } + x _{ 1,\ell } \right)
    ^4 \right] ,
\end{align*} 
which considerably simplifies the form of the fast quadratic
potential.

Following the example treated numerically
by~\citet{HaLuWa2006,McON2007}, we consider an instance of the FPU
problem, integrated over the time interval $[0,200]$, with parameters
$ \ell = 3 $, $ \omega = 50 $, whose initial conditions are
\begin{equation*}
x _{ 0, 1 } (0) = 1, \quad y _{ 0, 1 } (0) = 1 , \quad x _{ 1, 1 } (0)
= \omega ^{-1} , \quad y _{ 1,1 } (0) = 1 ,
\end{equation*}
with zero for all other initial values.  This displays an interesting
and complex property of the FPU problem, called {\em slow energy
  exchange}, which results from the slow nonlinear coupling between
the stiff springs.  If we consider only the energies in the stiff
springs, written as
\begin{equation*}
I _j \left( x _{ 1,j } , y _{ 1,j } \right) = \frac{1}{2} \left( y _{
    1, j } ^2 + \omega ^2 x _{ 1,j } ^2 \right) , \qquad j = 1, 2, 3,
\end{equation*}
then the initial conditions start with all of the energy in $ I _1 $
and none in $ I _2, I _3 $.  Over the course of the time interval,
this energy is transferred in a characteristic way from $ I _1 $ to $
I _3 $, gradually transitioning through the middle spring $ I _2 $.
Furthermore, the total stiff energy $ I = I _1 + I _2 + I _3 $ remains
nearly constant, i.e., is an adiabatic invariant of the system.

\begin{figure}
  \subfloat[Reference solution: St\"ormer/Verlet with time step size $
  h = 0.001 $.]{\includegraphics{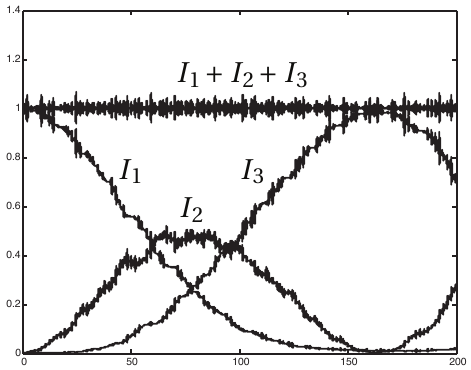}}
  \subfloat[St\"ormer/Verlet with $ h = 0.01
  $.]{\includegraphics{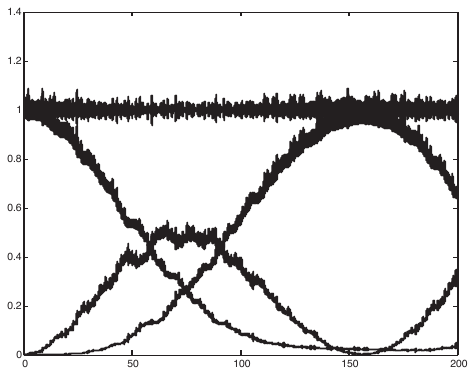}}
  \subfloat[St\"ormer/Verlet with $ h = 0.03 $.]{\includegraphics{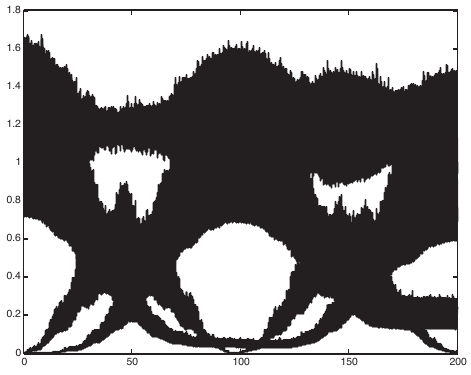}}\\
  \subfloat[IMEX with $ h = 0.03 $.]{\includegraphics{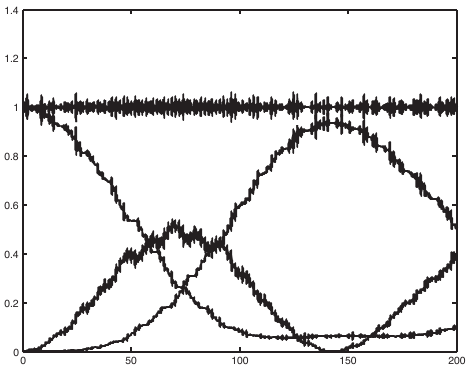}}
  \subfloat[IMEX with $ h = 0.1 $.]{\includegraphics{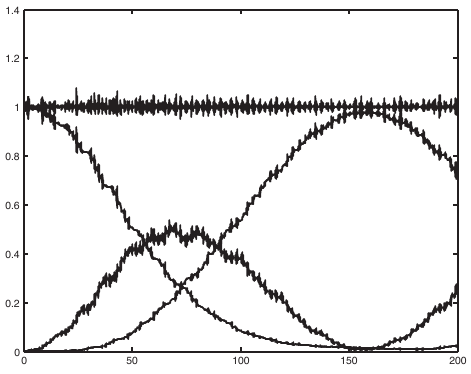}}
  \subfloat[IMEX with $ h = 0.15 $.]{\includegraphics{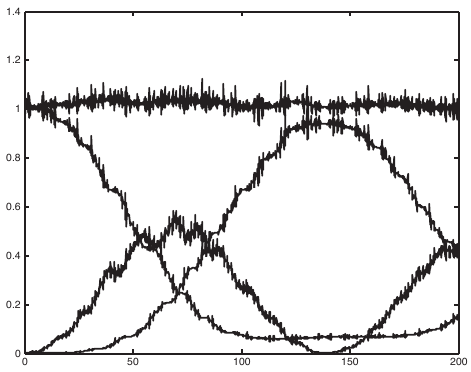}}\\
  \subfloat[IMEX with $ h = 0.2 $.]{\includegraphics{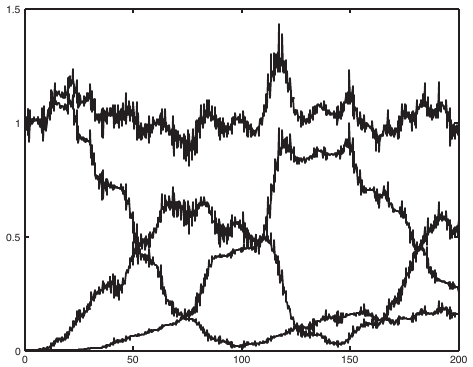}}
  \subfloat[IMEX with $ h = 0.25 $.]{\includegraphics{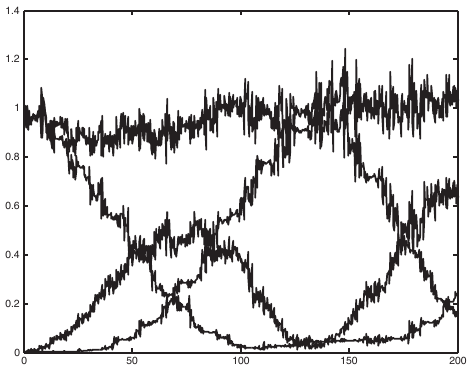}}
  \subfloat[IMEX with $ h = 0.3 $.]{\includegraphics{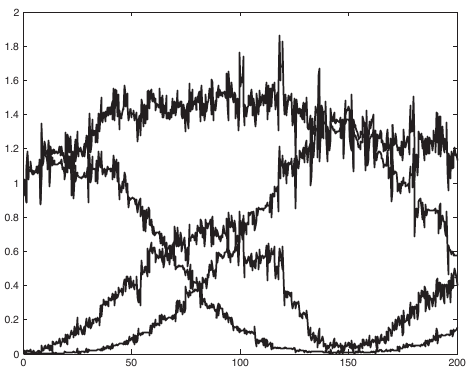}}
  \caption{The IMEX method robustly captures slow energy exchange in
    the Fermi--Pasta--Ulam problem with $ \omega = 50 $, even for large
    time steps.  Because the fast force is integrated implicitly, IMEX
    remains stable and degrades gradually as the time step size
    increases---unlike the fully explicit St\"ormer/Verlet method,
    which rapidly becomes unstable.}
  \label{fig:FPU}
\end{figure}

\autoref{fig:FPU} shows several numerical simulations of this FPU
energy exchange, computed both with St\"ormer/Verlet and with the
variational IMEX method, for different choices of time step size.  The
first plot is a reference solution, computed using St\"ormer/Verlet
with $ h = 0.001 $, fully resolving the fast oscillations.  However,
we see that the St\"ormer/Verlet solution's quality and stability
degrade rapidly as we increase the step size (for $ h = 0.03 $, we
have $ h \omega = 1.5 $, which is near the upper end of the stability
region $ \left\lvert h \omega \right\rvert \leq 2 $).  By contrast,
the variational IMEX method performs extremely well for $ h =
0.03\text{--}0.15$, degrading gradually as the time step size
increases.  Even as the numerical solution begins to undergo serious
degradation for $ h = 0.2\text{--}0.3 $, the qualitative structure of
the energy exchange behavior between $ I _1, I _2, I _3 $ is still
maintained.  Compare~\citealp[p. 24, Figure 5.3]{HaLuWa2006}; see
also~\citealp{McON2007}, who examine a wide variety of geometric
integrators, particularly trigonometric integrators, for the FPU
problem, with respect to both resonance stability and slow energy
exchange.  In particular, these authors found that the existing
trigonometric integrators exhibit a trade-off between correct energy
exchange behavior and resonance stability, and that these features
tend to be mutually exclusive.

\begin{figure}
  \subfloat[IMEX with $
  h=0.1$]{\includegraphics[width=.45\linewidth]{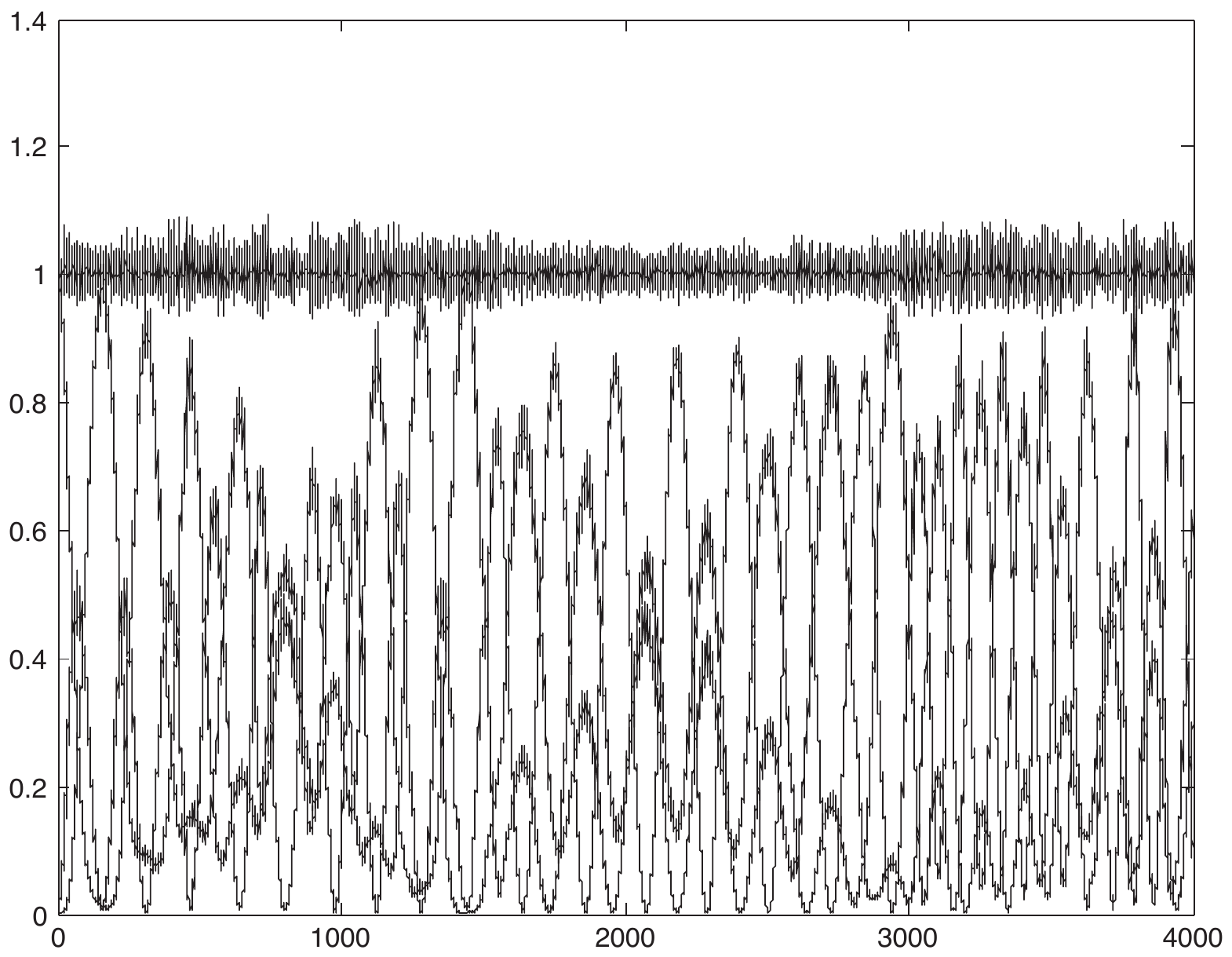}}
  \hfill \subfloat[IMEX with $ h=0.3
  $]{\includegraphics[width=.45\linewidth]{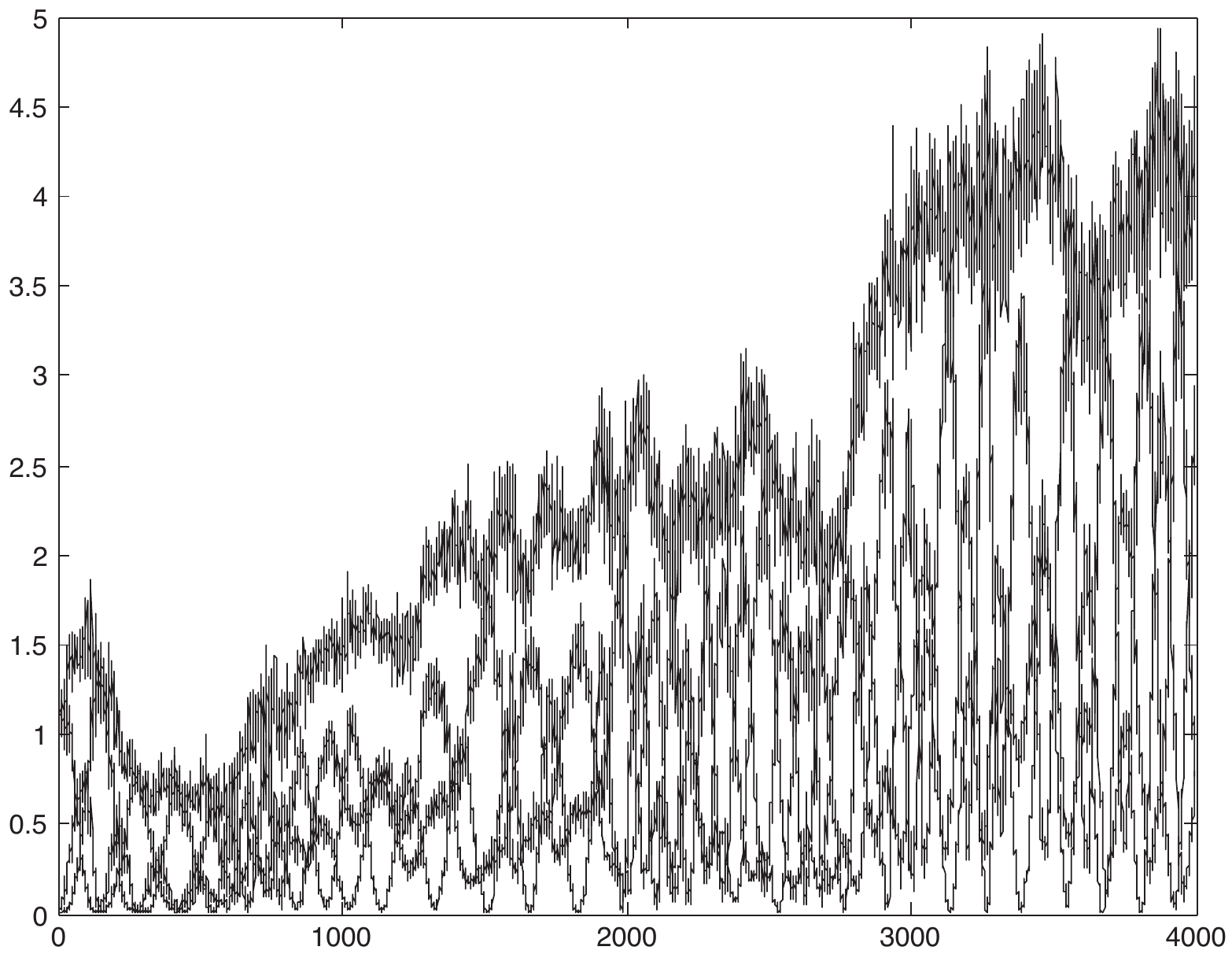}}
  \caption{Numerical simulation of the FPU problem for $ T = 4000 $,
    which shows the behavior of the IMEX method on the $ \omega ^2 $
    scale.  For $ h = 0.1 $, we already have $ h \omega = 5 $, yet the
    oscillatory behavior and adiabatic invariant are qualitatively
    correct.  By contrast, for $ h = 0.3 $, the method has begun to blow
    up; oscillatory coupling is a drawback of implicit midpoint
    methods for large time steps.}
  \label{fig:FPUlongtime}
\end{figure}

In \autoref{fig:FPUlongtime} we show the numerical behavior of the
variational IMEX method, applied to this same FPU problem, on a longer
time scale ($T=4000$) and for large time steps ($ h = 0.1, 0.3 $).  At
$ h = 0.1 $, the IMEX simulation still displays the correct
qualitative energy behavior, with respect to both the slow energy
exchange and the adiabatic invariant $I$, and the numerical solution
remains bounded.  However, by $ h = 0.3 $, numerical stability has
broken down, as oscillatory coupling in the fast modes leads to
unbounded amplitude growth.  This illustrates one of the drawbacks of
implicit midpoint-type methods: despite the lack of linear resonances,
numerical instability can still result for very large time steps due
to nonlinear coupling \citep{AsRe1999,AsRe1999a}.

This example was chosen to demonstrate that the variational IMEX
method does not attain its stability merely by ``smoothing out'' the
fast frequencies, in a way that might destroy the structure of any
fast-slow nonlinear coupling.  Rather, despite the fact that it does
not resolve the fast frequencies, the method is still capable of
capturing the complex multiscale interactions seen in the FPU
problem.

\section{Analysis of Slow Energy Exchange in the IMEX Method}
\label{sec:see}

In the previous section, the numerical experiments for the
Fermi--Pasta--Ulam problem seemed to suggest that the variational IMEX
method preserves the slow energy exchange between the fast oscillatory
modes.  This is somewhat surprising, since the method does not
actually resolve these fast oscillations.  However, in this section,
we will prove that, in fact, this method {\em does} accurately
reproduce the slow energy exchange behavior, as long as the numerical
solutions remain bounded.  This is demonstrated by showing that the
variational IMEX method can be understood as a {\em modified impulse
  method}; that is, the midpoint step exactly resolves the
oscillations of some modified differential equation.  We can then
apply some of the existing theory about numerical energy exchange for
impulse methods.  (It should be noted that impulse methods, which
originated with the work of~\citealp{Deuflhard1979}, can be understood
as a special case of trigonometric integrators when applied to highly
oscillatory problems.)

First, let us rewrite the fast oscillatory system $ \ddot{q} + \Omega
^2 q = 0 $ as the first-order system
\begin{equation*}
  \begin{pmatrix}
    \Omega \dot{q} \\
    \dot{p} 
  \end{pmatrix} = 
  \begin{pmatrix}
    0 & \Omega \\
    - \Omega  & 0
  \end{pmatrix}
  \begin{pmatrix}
    \Omega q \\
    p
  \end{pmatrix},
\end{equation*}
so it follows that the exact solution satisfies
\begin{equation*}
  \begin{pmatrix}
    \Omega q (t+h) \\
    p (t+h)
  \end{pmatrix} =
  \begin{pmatrix}
    \cos \left( h \Omega \right)  & \sin \left( h \Omega \right) \\
    - \sin \left( h \Omega \right) & \cos \left( h \Omega \right)
  \end{pmatrix}
  \begin{pmatrix}
    \Omega q (t) \\
    p (t)
  \end{pmatrix} .
\end{equation*}
We will now show that the implicit midpoint method effectively
replaces this rotation matrix for $ \Omega $ by the rotation matrix
corresponding to a modified $ \tilde{ \Omega } $.  In the transformed
coordinates just introduced, the implicit midpoint method has the
expression
\begin{equation*}
  \begin{pmatrix}
    I & - h \Omega / 2 \\
    h \Omega / 2 & I
  \end{pmatrix}
  \begin{pmatrix}
    \Omega q _{ n + 1 } \\
    p _{ n + 1 }
  \end{pmatrix} =
  \begin{pmatrix}
    I & h \Omega /2 \\
    - h \Omega / 2 & I 
  \end{pmatrix}
  \begin{pmatrix}
    \Omega q _n \\
    p _n 
  \end{pmatrix}.
\end{equation*}
Therefore, if we take the skew matrix
\begin{equation*}
A =
\begin{pmatrix}
0 &  \Omega \\
-\Omega & 0
\end{pmatrix},
\end{equation*}
it follows that 
\begin{equation*}
  \begin{pmatrix}
    \Omega q _{ n + 1 } \\
    p _{ n + 1 }
  \end{pmatrix} = \left( I - hA/2 \right) ^{-1} \left( I + hA/2 \right)
  \begin{pmatrix}
    \Omega q _n \\
    p _n 
  \end{pmatrix}.
\end{equation*}
Notice that the expression $ \left( I - hA/2 \right) ^{-1} \left( I +
  hA/2 \right) = \operatorname{cay}(hA)$ is the {\em Cayley
  transform}, which maps skew matrices to special orthogonal matrices
(and can be seen as an approximation to the matrix exponential map,
which gives the exact solution).  Hence the stability matrix is
special orthogonal, so we can write
\begin{equation*}
  \begin{pmatrix}
    \Omega q _{ n + 1 }  \\
    p _{ n + 1 } 
  \end{pmatrix} =
  \begin{pmatrix}
    \cos \left( h \tilde{\Omega} \right)  & \sin \left( h \tilde{\Omega} \right) \\
    - \sin \left( h \tilde{\Omega} \right) & \cos \left( h
      \tilde{\Omega} \right)
  \end{pmatrix}
  \begin{pmatrix}
    \Omega q _n  \\
    p _n
  \end{pmatrix} 
\end{equation*}
for some modified frequency $ \tilde{\Omega} $.  Therefore, the
stability matrix for the implicit midpoint method corresponds to the
{\em exact} flow matrix for a modified oscillatory system.

As an example, suppose we have $ \Omega =
\left( \begin{smallmatrix}
  0 & 0\\
  0 & \omega I 
\end{smallmatrix} \right) $ for some constant frequency $\omega$.
Applying the Cayley transform, it can be seen that the modified
frequency $ \tilde{\omega} $ satisfies
\begin{equation*}
h \omega / 2 = \tan \left( h \tilde{\omega} / 2 \right) .
\end{equation*}
Squaring both sides, this becomes
\begin{equation*}
  \left( h \omega / 2 \right) ^2 = \tan ^2 \left( h \tilde{\omega} / 2
  \right) = \frac{1 - \cos \left( h \tilde{\omega} \right)  }{ 1 + \cos
    \left( h \tilde{\omega} \right) } ,
\end{equation*}
which finally gives the solution for the modified frequency,
\begin{equation*}
\tilde{\omega} = \frac{ 1 }{ h } \arccos \left( \frac{ 1 - \left( h \omega / 2
    \right) ^2 }{ 1 + \left( h \omega / 2 \right) ^2 } \right) .
\end{equation*}

\begin{remark}
  This perspective provides another explanation as to why the
  variational IMEX method does not exhibit resonance: we always have $
  h \tilde{\omega} < \pi $.  In fact, the Cayley transform does not
  map to a rotation by $ \pi $, except in the limit as $ h \omega
  \rightarrow \infty $.  Therefore, for any finite $h$ and $ \omega $,
  we will never encounter the resonance points corresponding to
  integer multiples of $\pi$.

  As an aside, this also leads to another possible interpretation for
  the onset of nonlinear instability, if the time step size $h$
  becomes too large (as we saw in~\autoref{fig:FPUlongtime}).  Since $
  \tilde{\omega} < \pi / h $, the modified frequency $ \tilde{\omega}
  $ must shrink as $h$ grows.  Informally, then, if $ \tilde{\omega} $
  is very small, this can be seen as leading to amplitude growth in
  the fast modes, since it requires less energy to induce this
  amplification.
\end{remark}

Since the implicit midpoint method has now been seen as the exact
solution of a modified system, we can write the variational IMEX
method as the following modified impulse scheme:
\begin{alignat*}{2}
  \textbf{Step 1:}&& p _n ^+ &= p _n - \frac{ h }{ 2 } \nabla U \left(
    q _n \right), \\
  \textbf{Step 2:} &\quad&
  \begin{pmatrix}
    \Omega q _{ n + 1 } \\
    p _{ n + 1 } ^-
  \end{pmatrix} &=
  \begin{pmatrix}
    \cos \left( h \tilde{\Omega} \right)  & \sin \left( h \tilde{\Omega} \right) \\
    - \sin \left( h \tilde{\Omega} \right) & \cos \left( h
      \tilde{\Omega} \right)
  \end{pmatrix}
  \begin{pmatrix}
    \Omega q _n \\
    p _n ^+
  \end{pmatrix} ,\footnotemark
  \\
  \textbf{Step 3:}&& p _{n+1} &= p _{n+1}^- - \frac{ h }{ 2 } \nabla U
  \left( q _{n+1} \right) .
\end{alignat*}
\footnotetext{Note that although Step 2 might appear to be
  ill-defined, due to the fact that $\Omega$ is possibly singular, the
  singularity can be removed by substituting the relation $ h \Omega /
  2 = \tan \left( h \tilde{ \Omega } / 2 \right) $.  The explicit
  equation for $ q _{ n + 1 } $ and $ p _{ n + 1 } ^- $ is then
  calculated to be
\begin{equation*}
  \begin{pmatrix}
    q _{ n + 1 }  \\
    p _{ n + 1 } ^-
  \end{pmatrix} =
  \begin{pmatrix}
    \cos \left( h \tilde{\Omega} \right) & h/2 \left( 1 +
      \cos \left( h \tilde{ \Omega } \right) \right)  \\
    - 2/h  \left( 1 - \cos \left( h \tilde{ \Omega } \right) \right) &
    \cos \left( h \tilde{\Omega} \right)
  \end{pmatrix}
  \begin{pmatrix}
    q _n  \\
    p _n ^+
  \end{pmatrix} ,
\end{equation*}
which is seen to recover the correct, purely kinematic equation when $
\Omega = \tilde{ \Omega } = 0 $.
}Suppose again that  $ \Omega =
\left( \begin{smallmatrix}
  0 & 0\\
  0 & \omega I 
\end{smallmatrix} \right) $ for some constant frequency $\omega$, and
likewise $ \tilde{\Omega} = \left( \begin{smallmatrix}
    0 & 0\\
    0 & \tilde{\omega} I
  \end{smallmatrix} \right) $.  (This includes the case of the FPU
problem.)  The slow energy exchange behavior of this system was
analyzed in detail by \citet[XIII, see especially p. 495]{HaLuWa2006}
using the so-called {\em modulated Fourier expansion}; we now give a
brief, high-level summary of this work.  In the notation of
\citeauthor{HaLuWa2006}, the exact solution is asymptotically expanded
as $ x _\ast (t) \sim q (t) $, where
\begin{equation*}
  x _\ast (t) = y (t) + e ^{ i \omega t } z (t) + e ^{ - i \omega t } \bar{z} (t) .
\end{equation*}
Here, $ y (t) $ is a smooth real-valued function and $ z (t) $ is a
smooth complex-valued function, and these can be partitioned according
to the blocks of $\Omega$ as $ y = \left( y _0, y _1 \right) $ and $ z
= \left( z _0, z _1 \right) $.  They show that, assuming the exact
solution has energy bounded independent of $\omega$, this implies $ z
(t) = \mathcal{O} \left( \omega ^{-1} \right) $, so $z$ describes the
slow-scale evolution of the system.  Plugging in this ansatz for a
highly oscillatory system, and eliminating the variables $ y _1 $ and
$ z _0 $, the slow evolution turns out to be described by
\begin{equation*}
  2 i \omega \dot{ z }_1  = \frac{ \partial g _1 }{ \partial x _1 }
  \left( y _0 , 0 \right) z _1 + \mathcal{O} \left( \omega ^{ - 3 }
  \right) ,
\end{equation*}
where here the slow force $ g (q) = - \nabla U (q) $ has also been
block-decomposed as $ g = \left( g _0, g _1 \right) $.

\Citeauthor{HaLuWa2006} compare the above with the numerical solution
for a trigonometric integrator, which is similarly expanded as
\begin{equation*}
  x _h (t) = y _h (t) + e ^{ i \omega t } z _h (t) + e ^{ -
    i \omega t } \bar{z} _h (t) ,
\end{equation*}
with $ y _h = \left( y _{h,0}, y _{ h, 1 } \right) $ and $ z _h =
\left( z _{ h, 0 } , z _{ h, 1 } \right) $.  For the unmodified
Deuflhard/impulse method, in particular, the slow-scale numerical
evolution is given by
\begin{equation*}
  2 i \omega \dot{ z }_{ h, 1 } = \frac{ \partial g _1 }{ \partial x _1
  } \left( y _{ h,0}, 0 \right) z _{ h, 1 } + \mathcal{O} \left(
    \omega ^{ - 3 } \right) ,
\end{equation*}
which implies that the equation for the impulse solution $ z _{ h, 1 }
$ is consistent with that for $ z _1 $.

We now finally have what we need to prove our main result on the slow
energy exchange behavior of the variational IMEX method.
\begin{theorem}
  Let the variational IMEX method be applied to the highly oscillatory
  problem above, and suppose the numerical solution remains bounded.
  Then the ordinary differential equation for $ \tilde{ z } _{ h, 1 }
  $, describing the slow energy exchange in the numerical solution, is
  consistent with that for $ z _1 $ in the exact solution; this holds
  up to order $ \mathcal{O} \left( \omega ^{-3} \right) $.
\end{theorem}

\begin{proof}
  As we have previously shown, the IMEX scheme corresponds to a
  modified impulse method with frequency $ \tilde{ \omega } $.
  Therefore, to get the equation for $ \tilde{ z } _{ h, 1 } $, we
  must modify the equation above for $ z _{ h , 1 } $ by replacing $
  \omega $ with $ \tilde{ \omega } $ on the left hand side.  However,
  notice that Step 2 of the modified method advances the {\em
    original} state vector $ \left( \begin{smallmatrix}
      \Omega q _n \\
      p _n
  \end{smallmatrix} \right) $, rather than the modified $
\left( \begin{smallmatrix}
    \tilde{\Omega} q _n \\
    p _n
\end{smallmatrix} \right) $.  Changing from $ \left( \begin{smallmatrix}
    \Omega q _n \\
    p _n
  \end{smallmatrix} \right) $ to $
\left( \begin{smallmatrix}
    \tilde{\Omega} q _n \\
    p _n
  \end{smallmatrix} \right) $ also introduces a scaling factor of $
\tilde{ \omega} /\omega $ on the right hand side.  Therefore, the
variational IMEX solution satisfies the slow-scale equation
\begin{equation*}
  2 i \tilde{ \omega } \dot{ \tilde{ z } }_{ h, 1 } = \frac{ \tilde{
      \omega } }{ \omega } \frac{ \partial g _1 }{ \partial x _1 }
  \left( \tilde{ y } _{ h,0}, 0 \right) \tilde{ z } _{ h, 1 } + \mathcal{O} \left( \omega  ^{ - 4 } \tilde{ \omega } \right) .
\end{equation*}
Finally, cancelling the $ \tilde{ \omega } $ factors and multiplying
by $\omega$, we once again get
\begin{equation*}
  2 i \omega \dot{ \tilde{ z} }_{ h, 1 } = \frac{ \partial g _1 }{ \partial x _1
  } \left( \tilde{ y } _{ h,0}, 0 \right) \tilde{ z } _{ h, 1 } + \mathcal{O} \left( \omega ^{ - 3 } \right) ,
\end{equation*}
which is the same as that for the coefficient $ z _{ h, 1 } $ in the
original, unmodified impulse method.  This completes the proof.
\end{proof}

\section*{Acknowledgments}
The authors would like to thank Will Fong and Adrian Lew for helpful
conversations about the stability of AVI methods, Houman Owhadi for
suggesting that we examine the FPU problem, Dion O'Neale for advice
regarding the FPU simulation plots, Teng Zhang for pointing out a bug
in an earlier version of our FPU simulation code, and Jerry Marsden
and Reinout Quispel for their valuable suggestions and feedback about
this work.  We especially wish to acknowledge Marlis Hochbruck, Arieh
Iserles, and Christian Lubich for suggesting that we try to understand
the variational IMEX scheme as a modified impulse method; this
perspective led us directly to the results in~\autoref{sec:see}.
Finally, this paper benefited greatly from the thoughtful critiques
and suggestions of the anonymous referees, whom we also wish to thank.

\bibliographystyle{sternurl}
\bibliography{masterdb_stern}

\end{document}